\documentclass[11pt]{amsproc}
\usepackage{amstext,amsmath,amssymb,amsthm, esint}
\usepackage{latexsym}
\usepackage{exscale}
\usepackage{color}

\RequirePackage{srcltx}

\RequirePackage[cp1251]{inputenc}

\numberwithin{equation}{section}

\renewcommand\r{\rangle}
\renewcommand\l{\langle}

\newcommand\dsize{\displaystyle}

\renewcommand\Re{\operatorname{Re}}
\renewcommand\Im{\operatorname{Im}}

\newcommand\cal{\mathcal}

\newcommand\R{\mathbb{R}}

\newcommand\Z{\mathbb{Z}}

\newtheorem{Thm}{Theorem}[section]
\newtheorem{Lemma}[Thm]{Lemma}
\newtheorem{Cor}[Thm]{Corollary}

\theoremstyle{remark}

\renewcommand\hat{\widehat}
\begin{document}

\title[One-parameter groups   and discrete   Hilbert transforms]{One-parameter groups of  operators   and discrete   Hilbert transforms}
\author{Laura De Carli}
\address{L.~De Carli, Florida International University,
Department of Mathematics,
Miami, FL 33199, USA}
\email{decarlil@fiu.edu}
\author{Gohin Shaikh Samad}
\address{ Gohin Shaikh Samad, Univ. Iowa, Department of Mathematics
14 MacLean Hall
Iowa City, IA 52242-1419}
\email{shaikhgohin-samad@uiowa.com}
 \maketitle
 \begin{abstract}  We    show that the  discrete  Hilbert  transform   and the discrete  Kak-Hilbert transform 
 are infinitesimal generator of     one-parameter  groups   of operators in $\ell^2$.   
 \end{abstract}

 \section{Introduction}
We are concerned with  the family of operators $\{T_t \}_{t\in\R}$,   initially defined in the space $s_0$ of  complex-valued sequences with compact support   as follows: 
\begin{equation}\label{e1-defT}
(T_t(\vec a))_m= \begin{cases}  \dsize \frac{ \sin ( \pi t)}{ \pi}\sum_{n\in\Z} \frac{a_n  }{m-   n  +t} 
& \mbox{ if $t\not\in\Z$}
\cr  
 \dsize   (-1)^t a_{m+t} & \mbox{if $t \in\Z$}. \end{cases}
\end{equation}
When $t$  is an integer,   $T_t(\vec a)= (-1)^t\tau_t(\vec a)$, where $\tau_k(\vec a)_m= a_{k+m}$ is the translation;
when $t   \in(-1,1)$  these operators    can be viewed as discrete versions of the     Hilbert transform in $L^2(\R)$.

The Hilbert transform  $${\cal H}f(x)= p.v. \frac 1{\pi}\int_{-\infty}^\infty \frac{f(t)}{x-t}dt ,$$ initially defined when  $f\in C^\infty_0(\R)$, is the archetypal  singular integral operator.   
Discrete analogs of the Hilbert transform have important applications in
 science and technology.
The following operator   was    introduced by D. Hilbert in 1909.
 \begin{equation}\label{e1-def-DHT}(H(\vec a))_m= \frac 1 \pi \sum_{n\in\Z\atop n\ne m} \frac{a_n  }{m- {  n}  }.
\end{equation}
This transformation    is   not well suited for  the applications for reasons that we will discuss in Section 2;      
the   operators  $$T_{\frac 12}(\vec a) =\frac 1\pi \sum_{n\in\Z}\frac{a_n}{m-n+\frac 12}$$ (E.C. Titchmarsh, 1926)  and   the {\it Kak-Hilbert transform}  (S. Kak, 1970, \cite{K}) 
\begin{equation}\label{e1-def-Kak}  K(\vec a)(k)= \begin{cases} \dsize \frac 2\pi\sum_{\mbox{\small $n$ even} }\frac{a_n}{k-n} & \mbox{k odd}
\\
\dsize \frac 2\pi\sum_{\mbox{\small $n$ odd}} \frac{a_n}{k-n} & \mbox{k even}
\end{cases} 
\end{equation}
share  some of the features of the continuous Hilbert transform and are very relevant in sciences and   engineering.
 We will discuss    these operators  in Sections 2 and 3.  
Weighted   discrete Hilbert transforms   and their connections with problems   in complex analysis are  discussed in \cite{BMS}.

\medskip
When $1\leq p<\infty$, we denote with $\ell^p$    the space of complex-valued p-summable sequences, i.e., 
$$\ell^p=\left\{\vec a= (a_j)_{j\in\Z} \ : \ ||\vec a||_{\ell^p}=\left(\sum_{j\in\Z} |a_j|^p\right)^{\frac 1p} <\infty\right\}.
$$
$\ell^\infty$ is the space of bounded sequences equipped with the norm $||\vec a||_{\ell^\infty}=\sup_{m\in\Z}|a_m|$.

Our main result  is the following
\begin{Thm}\label{T1-semigr-H}
The family $\{T_t\}_{t\ge 0}$ defined in \eqref{e1-defT} is a strongly continuous   group   of isometries  in $ \ell^2 $;  
   its  infinitesimal generator  is  
$\pi H$, where $H$ is  the operator defined  in 
\eqref{e1-def-DHT}.  
\end{Thm}
To prove  Theorem \ref{T1-semigr-H}    we will prove   that  $ T_s\circ T_t=T_{s+t};
$
(Theorem \ref{T4- composition-TdTs}); that  $T_t$ is an isometry for every $t\in\R$ (Theorem \ref{T4-norm-Tt}); that for every $ \vec a\in \ell^2$,  the application  $t\to  T_t(\vec a)$   is continuous in $\R$ (Theorem \ref{T4-convergence});    %
 and finally  that, for every $\vec a\in \ell^2$,   
$
\lim_{t\to 0} \frac{T_t(\vec a)- \vec a}{t}=\pi H(\vec a) 
$
(Theorem \ref{T4-inf-gen-H}). 
 
The   proofs  of these results  are   elementary and use  only the identities in Section 2.3. 
Theorem \ref{T4-norm-Tt} seem to be known, but we could not find references in the literature.
Some of the results in Section 4   can   also be  proved   in the framework of the theory of  Toeplitz operators \footnote{We are indebted to   I. Verbitsky for this remark}.

 \medskip
In Section 3 we describe  the properties of the Kak-Hilbert  transform and 
we prove the following
\begin{Thm}\label{T1-semigr-K} 
 Let $K$ be the discrete Kak-Hilbert  transform \eqref{e1-def-Kak}.  Then
$$ U_t= \cos t I+\sin   t K  =\Im(e^{-it}(I+iK)), \quad \quad t\in\R.
$$  
is a strongly continuous    group of operators in $\ell^2$ generated by $K$.
\end{Thm}

In Sections   2.1  and   3 we  discuss  the  $\ell^p-\ell^p$ boundness of the operators $H$, $T_t$ and $K$  for $1<p<\infty$.
It is noted in \cite{L}     that    the operators $T_t$  and $H$  (and in general,    every operator  $L:\ell^p\to\ell^p$ in the form of  $L(\vec a)_m=\sum_n a_{m-n} c_n$, with $(c_n)_{n\in\Z}\in \ell^\infty$)    can   be associated to  a Fourier multiplier operator   acting on functions on the real line.  Indeed, we can associate   to $L$   the operator $\tilde L : L^p(\R)\to L^p(\R)$
$$
\tilde L f(x)=  \sum_{n\in\Z} f(x-n)c_n=\int_\R \hat f(y) m (y)e^{2\pi i xy}dy
$$
where  $m (y)$ is the periodic function  on $\R$ whose Fourier  
coefficients are the $c_n$'s, and $\hat f(y)=\int_\R f(x)e^{2\pi i xy}dx$ is the Fourier transform of $f(x)$. 
For example, it is not too difficult to verify that the multiplier associated to the Kak-Hilbert transform \eqref{e1-def-Kak} is the "square wave"  function that coincides with 
 $m(x)=i\ \mbox{sgn}(x)
 $ 
 in $(-\frac 12, \frac 12)$.
 
It is proved in \cite{L} that  the $\ell^p\to \ell^p$ operator norm  of   $L$  is the same as the $L^p(\R) \to L^p(\R) $ operator norm of $\tilde L$.  In short, \begin{equation}\label{e1-equiv-p-norms}|||L|||_{ \ell^p }= |||\tilde L|||_{L^p }  .\end{equation}  Since the $L^2(\R)\to L^2(\R)$  norm  of a multiplier operator is the $L^\infty(\R)$ norm of the multiplier (see e.g. \cite{SW}) from \eqref{e1-equiv-p-norms} follows that
$ 
|||L|||_{\ell^2}=  \sup_{x\in\R}|m(x)|.
$
The evaluation of the $L^p(\R)\to L^p(\R)$ norm of multiplier operators is  often a very difficult problem, but the equivalence 
   \eqref{e1-equiv-p-norms} can  be used to produce an upper bound for the $\ell^p\to\ell^p$ norm of $L$. 
  Indeed, we will  prove in Section 3  (Theorems \ref{T3-p-norm-tildeK} and \ref{T3-p-norm-K})  that $|||{\cal H}|||_{L^p}\leq |||H|||_{\ell^p}  \leq |||K|||_{\ell^p}=|||T_{\frac 12}|||_{\ell^p}$.  We conjecture that these norms are equal for all values of $p\in (1, \infty)$.

\medskip

 {\bf Acknowledgement}. The first author wishes to 
  thank   L. Grafakos, E. Laeng  and S. Montgomery-Smith for useful discussions during her visit at the University of  Missouri.

\medskip

\section{Preliminaries}

\subsection{The   Hilbert transform}
The Hilbert transform  $${\cal H}f(x)= p.v. \frac 1{\pi}\int_{-\infty}^\infty \frac{f(t)}{x-t}dt,$$ initially defined when  $f\in C^\infty_0(\R)$, is an important  singular integral operator. We  refer the reader to the excellent   \cite{GrBook} for an introduction to the Hilbert transform and singular integrals.
 
The Hilbert transform satisfies the identity  ${\cal H\circ \cal H}( f)= -f$,  which implies that ${\cal H}$   is an isometry in  $L^2(\R)$.

When $f$ is real-valued,   $ f+i{\cal H}f$ extends to an holomorphic function in the upper  complex   half-plane.
This fundamental property of the Hilbert transform  has been used by S. Pichoridis \cite {P}  to evaluate the best constant in the inequality of M. Riesz:
\begin{equation}\label{e2-Lp-norm-Hilbert}
||{\cal H}f||_{L^p(\R )}\leq n_p||f||_{L^p(\R )}, \quad f\in C^\infty_0(\R).
\end{equation}
Here, $1<p<\infty$, and   $n_p=\max\{\tan(\pi/2p), \ \cot(\pi/2p)\}$. 
See also  \cite{G2} for a  short proof of Pichoridis' result.

\medskip
Discrete versions of the   Hilbert transform  have   a variety of applications in signal
representation and processing.  See e.g. \cite{SS} and the references cited there.    

To the best of our knowledge, the $\ell^2\to\ell^2$ norm of the operator $H $     defined in \eqref{e1-def-DHT}  has been estimated for the  first time by D. Hilbert 
who in 1909  proved  the inequality 
\begin{equation}\label{e1-ell2-norm-DHT}   \sum_{m\in\Z} \sum_{n\in\Z\atop n\ne m} \frac{a_n  b_m}{m- {  n}  } \leq c\left(\sum_{n\in\Z}|a_n|^2\right)^{\frac 12}\left(\sum_{m\in\Z}|b_m|^2\right)^{\frac 12}
\end{equation}
with a constant $c>\pi$. \footnote{The original proof  first appeared in Weyl's \cite{W} doctoral dissertation in 1908.} Three years later, Shur \cite{S}  proved that $c=\pi$ is the  best possible constant in the inequality \eqref{e1-ell2-norm-DHT}, or equivalently  that   $1$ is the $\ell^2\to\ell^2$ operator norm of $H$.
See \cite{G} for  an   elegant  elementary proof of Shur's inequality.

$H$  is not an isometry in  $\ell^2$.   Indeed,  the proof in  \cite{G}   shows  that   the equality $||H(\vec a)||_ {\ell^2}=||\vec a||_ {\ell^2}$ only holds  when $\vec a=0$.    Also, it is not true in general that $H\circ H (\vec a)  =-\vec a$.

\medskip
The operators $T_{\frac 12}$  is  a good  analog of the continuous Hilbert transform.  By Theorem \ref{T1-semigr-H}, $T_{\frac 12} $    is an isometry in $\ell^2$ and  satisfies $T_{\frac 12}\circ T_{\frac  12}(\vec a) =-\tau_{1}(\vec a)$.

\medskip
The     Kak-Hilbert transform  defined in \eqref{e1-def-Kak}    can be viewed as a "reduced"   discrete Hilbert transform    \eqref{e1-def-DHT}: 
if   let $\chi_e:\ell^2\to \ell^2$  
 be such that  $\chi_e(\vec a)_n = a_n$ when $n$ is even and $\chi_e(\vec a)_n=0$ when $n$ is odd, and  we let  $\chi_o(\vec a)=\vec a-\chi_e(\vec a)$, we can easily verify that 
 \begin{equation}\label{e2-alt-def-K}
K(\vec a)= 2\left(\chi_o\circ H\circ \chi_e(\vec a)+\chi_e\circ H\circ \chi_o(\vec a)\right).
\end{equation} 
S. Kak proved in \cite{K} that $K\circ K(\vec a)= - \vec a $, from which follows that    $K$ is an isometry in $\ell^2$. 

 It is proved in \cite{HLP}   that $H$ is bounded in $\ell^p$ for $1<p<\infty$, and in   \cite{L}, Theorem 4.3, that the   $\ell^p\to \ell^p$ operator norm    of $H$    is  $\ge   n_p$, where  $n_p $ is the constant  in \eqref{e2-Lp-norm-Hilbert}.    Equality is proved   for  special values of $p$.  

 It is  also proved in \cite{L} that   the operators $T_t$ are bounded in $\ell^p$ for $1<p<\infty$,  and 
 $ |||T_t|||_{\ell^p}\ge ||| \cos(\pi t) {\cal I} + \sin(\pi t) {\cal H}|||_{L^p}$, where ${\cal I}f=f $.
The evaluation of the $\ell^p\to\ell^p$ operator norms of $H$ and $T_t$ is a tantalizing long-standing open problem. 
 %

\subsection{Groups and semigroups of operators}    Let $ X $ be a Banach space with norm $||\ ||$ and let ${\cal L}(X)$ be the collection of linear and bounded operators on $X$.
A {\it one parameter group of operators} is a mapping  $U:\R \to {\mathcal L} (X)$  such that 
(a)  $U(0)$ is the identity operator in ${\mathcal L} (X)$  and  (b)   $U(s)\circ  U(t)=U(s+t)$ whenever $s,\  t\in\R$. In particular  $U(-s)=U^{-1}(s)$.

A {\it  semigroup } is a mapping   $U :[0,\,\infty)\to {\mathcal L} (X)$ that satisfies  (a) and   (b)  whenever $s,\ t\ge 0$. 

We say that  a group (or semigroup)  $U$ is {\it strongly continuous}  if $\lim_{t\to t_0} ||U(t)(x)- U(t_0)(x)||=0$   for every $x\in X$. When $U$ is a semigroup, we also require that  
 $\lim_{t\to 0+} ||U(t)(x)-  x||=0$.
 We say that $U$ is {\it contractive} if $||U(t)(x)||\leq   ||x||$ for every    $x\in X$ and every $t\in\R$  (or: for every $t\ge 0$ if $U$ is a semigroup).


 The {\it infinitesimal generator} $A$ of a strongly continuous group (or semigroup) $U(t)$ can be introduced as the operator defined by 
\begin{equation}\label{def-inf-gen}
 A(x)=:  \frac {d}{dt}U\vert_{t=0}=  \lim_{h\to 0^+} \frac{U(h)(x)-x}{h}, \qquad x\in  D(A)
 \end{equation}
 where $D(A)$ is the set of all $x\in X$ for which the above limit  exists.   Using the strong continuity of $U(t)$, it is possible to prove that   $D(A)$ is dense in $X$.  It can  also be proved  that the  equation below is valid for every $x\in D(A)$:
\begin{equation}\label{def-exp-inf-gen}
 U(t)(x)=:  e^{t A(x)}= \sum_{n=0}^\infty \frac{A^{(n)}(x) t^n}{n!}
 \end{equation} 
 where $A^{(n)}$ denotes the iterated compositions of $A$.  
 
 The Hille-Yosida theorem gives necessary and sufficient conditions for  an operator  $A$  whose  domain  is dense in $X$ to be the infinitesimal generators of a contractive semigroup. 
\begin{Thm}\label{T2-Hille-Y}
Let $A$ be a linear operator defined on a linear subspace $ D(A)$ of a Banach space $X$. Then, A is the infinitesimal generator of a  contractive semigroup  if and only if 

 i) $A-\lambda I $ is invertible for every  $\lambda \in (0, \infty)$, 
 
 ii) $||(A-\lambda I)^{-1}(x)||<\frac {||x||} \lambda $ for every $x\in D(A)$ and $\lambda>0$.
\end{Thm}

The Hille-Yosida theorem  is fundamental in  the applications   to partial differential equations: indeed, if 
 $A$ is the infinitesimal generator of   a 
semigroup (or group) $U(t)$  in a Banach    space $X$, the vector
function $u(t)= U(t)( u_0)$ solves the abstract initial value problem
$$
\begin{cases} u'(t)= A u(t) \quad  t>0\ (t\in\R)
\cr u(0)= u_0
\end{cases}
$$
for any given initial value $u_0 \in D(A)$.   
 
 We  refer the reader to  the classical textbooks  \cite{HP} and  \cite{Y} for more applications  and results.

\medskip
\subsection{Partial fraction decomposition}  
The partial fraction expression of the cotangent function was proven by Euler in his {\it Introductio
in Analysis  Infinitorum (1748)}  for every  non-integer  $x$, and is regarded as one of the most interesting
formula involving elementary functions:
 \begin{equation}\label{e2-cotangent}
 \pi \cot(\pi x)= \frac 1x+ \sum_{n=1}^\infty \frac{1}{x+n}+\frac{1}{x-n}. 
 \end{equation}
 An elegant proof of this identity can be found in \cite{AZ}, pg. 149.
Using \eqref{e2-cotangent}, we can easily prove the following identity, which is valid  for every  non-integers $u, \ v\in \R$, with $u\ne v$,
 \begin{equation}\label{e2-sum-double-cot}
 \sum_{m=-\infty}^\infty \frac 1{(m-u)(m-v)}= \frac { \pi(\cot(\pi  v)-\cot(\pi u))} {u-v}.
\end{equation}
We will also use   the following  well known identities: when $d$ is not an integer,
\begin{equation}\label{e2-cosec}
\sum_{n\in\Z  }\frac{1}{(n+d)^2}=  \pi^2\csc^2 (d\pi) 
 \end{equation}
and when $d$ is an integer,
\begin{equation}\label{e2-sum-d-integer}
  \sum_{n\in\Z \atop{n\ne -d}}\frac{1}{(n+d)^2}=\frac{\pi^2}{3}. 
 \end{equation}

\section{The    Kak-Hilbert  transform}

\medskip
As recalled in   Section 2.1,   the Kak-Hilbert  transform  \eqref{e1-def-Kak}  shares a remarkable  number of properties with the continuous   Hilbert transform. Recalling the definition of $\chi_e$ and $\chi_o$ from section 2.1, we can easily verify that 
\begin{equation}\label{e3-K and T1}   K( \chi_o(\vec a))_{2m}= \frac 2\pi \sum_{n\in\Z}\frac{a_{2n+1}}{2m -2n-1 }= \frac 1\pi \sum_{n\in\Z}\frac{a_{2n+1}}{ (m -1) -n +\frac 12}
$$$$=
\tau_{-1}T_{\frac 12}(\tau_1\delta_2(\vec a))_m
\end{equation}
where we have left   $\delta_2((a_j)_{j\in\Z})=(a_{2j})_{j\in\Z}$ and $\tau_k ((b_j)_{j\in\Z})=(b_{j+k})_{j\in\Z}$.  Similarly, we prove that
\begin{equation}\label{e3-K and T2}
   K( \chi_e(\vec a))_{2m +1}=  T_{\frac 12}( \delta_2(\vec a))_m .
 \end{equation}
Note also that  $  K(\chi_e(\vec a))_h  = 0$ when   $h$ is even and $  K(\chi_o(\vec a))_h  = 0$ when   $h$ is odd.
Therefore, for every $p>0$
\begin{equation}\label{e3-prop-K}\sum_{m\in\Z}|  K(\vec a)_m|^p= \sum_{m\in\Z}|  K(\chi_e(\vec a) )_{2m}|^p+\sum_{m\in\Z}| K(\chi_o(\vec a))_{2m+1}|^p.
\end{equation}

We prove the following
 
 \begin{Thm}\label{T3-p-norm-K} 
 For every $1<p<\infty$  
 \begin{equation}\label{e3-p-norm-K}
|||K |||_{\ell^p}=|||T_{\frac 12}|||_{\ell^p}.
 \end{equation}
  \end{Thm}

 \begin{proof} 
Let $t_p$ be the  $\ell^p\to\ell^p$ operator norm of $T_{\frac 12}$. By \eqref{e3-K and T1}, \eqref{e3-K and T2} and \eqref{e3-prop-K}
\begin{align*}
||  K(\vec a)||_{\ell^p}^p &   = \sum_{m\in\Z}|  K(\chi_e(\vec a)))_{2m}|^p+\sum_{m\in\Z}|  K(\chi_o(\vec a))_{2m+1}|^p,
\\
&=\sum_{m\in\Z}|T_{\frac 12}(\delta_2( \vec a) )_{ m}|^p+\sum_{m\in\Z}|\tau_{-1} T_{\frac 12} ( \tau_1(\delta_2(\vec a)))_{m}|^p 
\\ &\leq t_p^p(||\delta_2( \vec a)||_{\ell^p}^p+||\tau_{ 1}(\delta_2( \vec a))||_{\ell^p}^p)= t_p^p ||\vec a||_{\ell^p}^p 
\end{align*}
from which follows that $|||  K |||_{\ell^p}\leq t_p$. 

Let us show  that $||  K(\vec a)||_{\ell^p}\ge t_p$. We  let   $ {\cal E}=\{\vec a\in\ell^p\ : \  \chi_o(\vec a)= \vec 0\}$ and observe that  for every $\vec a\in {\cal E}$, we have that  $||\vec a||_{\ell^p}=||\delta_2(\vec a)||_{\ell^p}$. In view of  $  K(\chi_e(\vec a))= T_{\frac 12} (\delta_2(\vec a))$, 
 we can write the following chain of inequalities:
 $$
 |||   K |||_{\ell^p} \ge \sup_{\vec a\in {\cal E}} \frac{||   K(\vec a)||_{\ell^p} }{||\vec a||_{\ell^p}}=
 \sup_{ \vec a \in {\cal E}} \frac{ || T_{\frac 12}( \delta_2(\vec a))  ||_{\ell^p}}{|| \delta_2(\vec a)   ||_{\ell^p}}
 =  \sup_{ \vec b \in\ell^p} \frac{ || T_{\frac 12}( \vec b)    ||_{\ell^p}}{|| \vec b  ||_{\ell^p}}=t_p.
$$
as required.

\end{proof}
 
We  let
\begin{equation}\label {e3- def-tildeK}\widetilde K(\vec a)_m =  (2 H-K)(\vec a)= 2\left(\chi_e\circ H\circ \chi_e(\vec a)+\chi_o\circ H\circ \chi_o(\vec a)\right).  \end{equation} 
Thus,
  $\widetilde K(\chi_e(\vec a))_h  = 0$ when   $h$ is odd and $\widetilde K(\chi_o(\vec a))_h  = 0$ when   $h$ is even, and  for every $p>0$
  \begin{equation}\label{e3-prop-tildeK}\sum_{m\in\Z}|\widetilde K(\vec a)_m|^p= \sum_{m\in\Z}|\widetilde K(\chi_e(\vec a ))_{2m}|^p+\sum_{m\in\Z}|\widetilde K(\chi_o(\vec a))_{2m+1}|^p.
\end{equation}
 We can  easily verify that
%
 \begin{equation}\label{e3-alt-def-tildeK} \widetilde K( \chi_e(\vec a))_{2m}= H(\delta_2(\vec a))_m, \quad  \widetilde K( \chi_o(\vec a))_{2m+1}= H(\tau_1\delta_2(\vec a))_m. 
 \end{equation}

 We prove the following
 
 \begin{Thm}\label{T3-p-norm-tildeK}
 For every $1<p<\infty$,
 \begin{equation}\label{e3-p-tildeK=H}||| \widetilde K |||_{\ell^p} = ||| H |||_{\ell^p}\end{equation}
 and 
 \begin{equation}\label{e3-p-K-and-H}||| K|||_{\ell^p}\ge ||| H |||_{\ell^p}.\end{equation}
 \end{Thm}
 \begin{proof}
  
The proof of \eqref{e3-p-tildeK=H} is similar to that of   \eqref{e3-p-norm-K}.
To prove   \eqref{e3-p-K-and-H}, we
observe   that $  K=2 H-\widetilde K$, and so
$$
|||K|||_{\ell^p}\ge \left|\,2|||H|||_{\ell^p}-|||\widetilde K|||_{\ell^p}\right|= |||H|||_{\ell^p}.
$$
 \end{proof}

 \medskip

 \noindent
 {\it Remark}. Recall that the $L^p(\R)-L^p(\R)$ operator norm of  the  Hilbert transform is the constant  $n_p$ defined in \eqref{e2-Lp-norm-Hilbert} and that $|||H |||_{\ell^p}\ge n_p$; 
 by Theorems \ref{T3-p-norm-tildeK}  and \ref{T3-p-norm-K}, 
 $$
 n_p\leq |||\widetilde K |||_{\ell^p}= |||H |||_{\ell^p}\leq  |||  K |||_{\ell^p}=|||T_{\frac 12}|||_{\ell^p}.
    $$
 It is conjectured in \cite{L} that   $|||T_{\frac 12}|||_{\ell^p}=n_p$. If this conjecture is proved, then also the   operator norms of $H$, $K$ and $\widetilde K$ equal $n_p$.
  \begin{proof}[Proof of Theorem \ref{T1-semigr-K}]
  The semigroup generated by  $K$  is  the operator   
  $ 
  U_t= e^{tK}= \sum_{n=0}^\infty \frac{t^n}{n!} K^{(n)} 
  $ 
  where $K^{(n)}$ is the n-times composition of $K$ with itself. 
  Recalling that $K\circ  K=-I$, we obtain
  $$
  \sum_{n=0}^\infty \frac{t^n}{n!} K^{(n)}= I+Kt- \frac{It^2}{2}-\frac{Kt^3}{3!}+...= I \cos t + K\sin t.
  $$
  as required.  \end{proof}
  \medskip
  \noindent
  {\it Remark.}  $U_t$ is not  an isometry in $\ell^2$. Indeed, 
\begin{align*}
||U_t(\vec a)||_{\ell^2}^2 &= || \vec a  \cos t + K(\vec a)\sin t||_{\ell^2}^2 \\ &= |\cos t |^2 ||\vec a||^2_{\ell^2} + |\sin t |^2 ||K(\vec a)||^2_{\ell^2}  +2\sin t \cos t\, \Re \l \vec a,\, K(\vec a)\r \\ & =||\vec a||^2_{\ell^2}+ \sin(2t)  \, \Re \l \vec a,\,   K(\vec a)\r
\end{align*}
  and we may have $\Re \l \vec a, \, K(\vec a)\r\ne 0$.

 \section{Proof of Theorem \ref{T1-semigr-H}}
 \medskip
 
  The proof of Theorem \ref{T1-semigr-H}   follows from several theorems and lemmas. Some of the results in this section  (in particular  Theorem \ref{T4-norm-Tt}) seem to be   known, but we could not find  proof  of these results in the literature.  
   
  \medskip
   First of all, we show that $ T_t $ is a semigroup.
  
  \begin{Thm}
\label{T4- composition-TdTs}
for every $s$, $t\in\R$   and for every $\vec a\in\ell^2$,
\begin{equation}\label{e3- TsTt=Ts+t}
T_d\circ T_s(\vec a)  =    T_{s+d}(\vec a).
 \end{equation}
In particular, $T_s^{-1}(\vec a)=T_{-s}(\vec a)$.
\end{Thm}

 \begin{proof} 
It is enough to prove the theorem for sequences $\vec a\in s_0$,  
 the space of complex-valued sequences with compact support,  because $s_0$ is dense in $\ell^2$.

 The identity \eqref{e3- TsTt=Ts+t} is clearly true when  $s$ and $d$ are both integers.  When $s$ is an integer 
 and $d$ is not integer \begin{align*}
 (T_sT_d  (\vec a))_k &=  (-1)^s\frac { \sin  ( \pi d)}{ \pi}
 \sum_{m\in\Z}\frac{a_m}{ (k+s)-m+d} \\ &=   \frac { \sin  (\pi (d+s))}{ \pi}
 \sum_{m\in\Z}\frac{a_m}{ k-m+(d+s)}= T_{d+s}(\vec a).
 \end{align*}
  
Suppose that   $s$, $d$  and $s+d$ are not integers; let $\vec a\in s_0$.
We can exchange the order of summation and make use  of  the identity \eqref{e2-sum-double-cot} to show that
\begin{align*}
 (T_dT_s(\vec a))_k  &=\frac{\sin  (\pi s)\sin  (\pi d)} {\pi^2}  \sum_{n\in\Z} a_n\sum_{  m\in\Z} \frac{1}{(k-m+d)(m-n+s)}      \\
=  & \frac{\sin  (\pi s)\sin  (\pi d)}{\pi}  \sum_{n\in\Z} a_n \frac{    \cot (\pi  (d-k)) -   \cot ( \pi (-n-s)) }{ n-k+d+s } 
 \\
 =  & \frac {\sin  (\pi s)\sin  (\pi d) }{\pi} \left(   \cot (\pi  d)+   \cot ( \pi s) \right)
\sum_{n\in\Z} \frac{a_n}{n-k+d+s } 
\\  = &
     \frac{\sin  (\pi (s+d)) }{\pi} \sum_{n\in\Z} \frac{a_n}{n-k+d+s }=T_{s+d}(\vec a).
\end{align*}
  as required.

\medskip
When $s$, $d$    are not integers and $s+d$ is an integer  the  proof is similar. 
\end{proof}

 \medskip
We prove  that  $T_t$ is strongly continuous. We start with the following
\begin{Lemma}\label{L4-Convergence-to-H}
Let $H$ be as in 
\eqref{e1-def-DHT}. For a given $d\in (-1, 1)$ and for every $\vec a\in \ell^2$, we let $\dsize H_d(\vec a)_m= \frac{1}{\pi} \sum_{n\in\Z\atop{n\ne m} }\frac{a_n}{n-m+d}$.
We have,
$$\lim_{d\to 0} ||H(\vec a)-H_d(\vec a)||_{\ell^2}=0.$$

\end{Lemma}
\begin{proof}
Observe  that
$
H_d(\vec a)_m- H(\vec a)_m   = -\frac{d}{\pi}\sum_{n\in\Z\atop{n\ne m}}  \frac{a_n}{(m-n+d)(m-n)}
 $
  is the convolution of $\vec a $ and $ \vec \nu=(\frac{1}{n(n+d)})_{n\in\Z\atop{n\ne 0}}$. 
  Using  the identity \eqref{e2-cotangent}   we can easily prove   that
   $||\vec \nu||_{\ell^1}=\sum_{n\ne 0} \frac 1{n(n+d)}=\frac{1-\pi  d \cot (\pi  d)}{d^2}  $. Furthermore, it is easy to verify that 
  $ \lim_{d\to 0} d||\vec \nu||_{\ell^1} = \lim_{d\to 0} \frac{1-\pi  d \cot (\pi  d)}{d }=0.$ 
  By Young inequality,
  $$
  \lim_{d\to 0}|| H_d(\vec a)  - H(\vec a)||_{\ell^2}\leq \frac 1\pi\lim_{d\to 0}\, d||\vec \nu||_{\ell^1} ||\vec a||_{\ell^2} 
=0 $$ 
as required.
\end{proof}

\medskip
\noindent
{\it Remark}. The proof of  Lemma \ref{L4-Convergence-to-H} shows that $H_d$ is bounded in $\ell^2$. Indeed,
\begin{align*}
||H_d(\vec a)||_{\ell^2} &\leq ||H_d(\vec a)-H(\vec a)||_{\ell^2}+||H(\vec a)||_{\ell^2} 
\\ &\leq \left(\frac{1-\pi  d \cot (\pi  d)}{\pi d } + 1\right)||\vec a ||_{\ell^2}. 
\end{align*}

\medskip
 We can now prove that $T_t$ is strongly continuous.
 \begin{Thm}\label{T4-convergence}
 
For every $t_0\in\R$ and every $\vec a\in\ell^2$,  
  $$\lim_{t\to t_0} ||T_t(\vec a)- T_{t_0}(\vec a)||_{\ell^2}=0. $$ 
  
 \end{Thm}
  
  \begin{proof}
   By Theorem \ref{T4- composition-TdTs},  $(T_t - T_{t_0})(\vec a)= T_ {t_0}(T_{t-t_0}-I)(\vec a)$, where $I$ is the identity in $\ell^2$. So,  in order to prove  the theorem we need only to prove that $
  \lim_{d\to 0} ||T_d(\vec a)-   \vec a ||_{\ell^2}=0$ for every $\vec a\in\ell^2$. Indeed, for every $m\in\Z$,
 \begin{align*}
  T_d(\vec a)_m-  a_m & = \frac{\sin(\pi d)}{\pi}\sum_{n\in\Z} \frac{a_n}{n-m+d} -a_m
  \\
  & =\sin(\pi d)  H_d(\vec a)_m+\frac{\sin(\pi d)}{\pi d} a_m -a_m    
  \end{align*}
  where $H_d$ is   as in Lemma \ref{L4-Convergence-to-H}. Thus,
  $$\lim_{d\to 0} ||  T_d(\vec a)-\vec a||_{\ell^2}\leq \lim_{d\to 0} \left(\frac{\sin (\pi d)}{\pi d}-1\right)||\vec a  ||_{\ell^2}+\lim_{d\to 0} \sin (\pi d)|| H_d(\vec a) ||_{\ell^2}=0.
  $$

 \end{proof}
 
 We denote with $T_t^*$ the adjoint operator of $T_t$;
in order to  prove that $T_t$ is an isometry,  we need  the following useful lemma: 

  \begin{Lemma}\label{L4-adj-Tt} For every $t\in\R$ and every $\vec a\in\ell^2$,  
   $$ T^*_t(\vec a)=T_{-t}(\vec a).$$

 \end{Lemma}
 \begin{proof} 
The lemma is trivial when   $t\in\Z$;  if $t\not\in\Z$,  and  $\vec a, \ \vec b\in \ell^2$,
\begin{align*}  
\l T_t (\vec a), \vec b\r =& \frac{\sin(\pi t)}{\pi}\sum_{m\in\Z} b_m\sum_{n\in\Z}\frac{a_n}{n-m+t}
\\
= &-\frac{\sin(\pi t)}{\pi}\sum_{n\in\Z} a_n\sum_{m\in\Z}\frac{b_m}{m-n-t} = \l \vec a,\  T_{-t}\vec b\r 
\end{align*}
as required.

 \end{proof}
 
 \begin{Thm}\label{T4-norm-Tt} For every  $t\in\R$ and every $a\in \ell^2$,
   $$||T_t (\vec a)||_{\ell^2}=   ||\vec  a||_{\ell^2}.$$   
 \end{Thm}
 
 \begin{proof} 
    By Lemma \ref{L4-adj-Tt} and Theorem \ref{T4- composition-TdTs},
    $$
    ||T_t(\vec a)||_{\ell^2}^2=\l T_t(\vec a), \ T_t (\vec a)\r= \l T_{-t} T_t(\vec a),\ \vec a\r= ||\vec  a||_{\ell^2}^2.
    $$
 \end{proof}
 
 We are left to prove that $\pi H$ is the infinitesimal generator of $T_t$.

 \begin{Thm}\label{T4-inf-gen-H}
 For every $\vec a\in \ell^2$,
 $$ \lim_{d\to 0}  \left\Vert\frac{ T_d(\vec a)-\vec a}{d}-\pi H(\vec a)\right\Vert_{\ell^2}=0.$$    
\end{Thm}

 \begin{proof} 
  We can write
\begin{align*}
  \frac{ ( T_d(\vec a)-\vec a)_m }{d} &=\frac{\sin   (\pi d)}{ \pi d } \sum_{n\in\Z}\frac{  a_n}{ m-n+d  } -\frac{a_m}{d}
\\
&=
 \frac{\sin   (\pi d)}{   d } H_d(\vec a)_m+\frac{a_m}{d} \left(\frac{\sin   (\pi d)}{ \pi d  }-1\right).
\end{align*}
Thus,
$$\left|\frac{ ( T_d(\vec a)-\vec a)_m }{d}-\pi H(\vec a)_m \right|
\leq\left\vert 
\frac{\sin   (\pi d)}{  d } H_d(\vec a)_m- \pi H(\vec a)_m \right|+ \left|\frac{ a_m }{d}\right|\, \left |1-\frac{\sin   (\pi d)}{ \pi d  } \right|
$$$$
\leq\left\vert 
\frac{\sin   (\pi d)}{   d } \right\vert | H_d(\vec a)_m- H(\vec a)_m  |  +\left\vert \frac{\sin   (\pi d)}{  d }-\pi \right\vert | H(\vec a)_m |+\left|\frac{ a_m }{d}\right|\, \left |1-\frac{\sin   (\pi d)}{ \pi d  } \right|,
$$
and for every $\vec a\in \ell^2$, we have
$$\Big\Vert\frac{   T_d(\vec a)-\vec a  }{d}  -\pi H(\vec a)\Big\Vert_{\ell^2}
 \leq  \left|\frac{\sin   (\pi d)}{  d } \right| \left\Vert H_d(\vec a)-H(\vec a) \right\Vert_{\ell^2}  
 $$$$ +  \left|\frac{\sin   (\pi d)}{   d }-\pi\right| \left\Vert H(\vec a) \right\Vert_{\ell^2} +\frac{||\vec a||_{\ell^2}}{d} \left |1-\frac{\sin   (\pi d)}{ \pi d  } \right|. 
$$ 
Since  $\lim_{d\to 0} \frac{\sin   (\pi d)}{   d }-\pi= \lim_{d\to 0}\frac 1d\left|\frac{\sin   (\pi d)}{ \pi d }-1\right|=0$,  the inequalities above and 
 Lemma \ref{L4-Convergence-to-H} yield 
  $\lim_{d\to 0}\left\Vert\frac{   T_d(\vec a)-\vec a  }{d}-\pi H(\vec a)\right\Vert_{\ell^2}=0$  as required.
 \end{proof}

\medskip

\subsection{Corollaries}

The following Corollaries easily follows from  Theorems \ref{T4- composition-TdTs} and \ref{T4-inf-gen-H} and   \eqref{def-inf-gen} and \eqref{def-exp-inf-gen}.

  \begin{Cor}\label{C3-composition-TsH}  let $T_t$ be as in \eqref{e1-defT} and $H$ as in \eqref{e1-def-DHT}.
  For every $s\in\R$, we have 
  \begin{itemize}
  \item[a)]
   $\dsize T_sH= H T_s=  \frac 1\pi {\frac{d}{dt}T_t \vert}_{t=s},  $
 
  \item[b)] $\dsize T_s= e^{s\pi H}= \sum_{k=0}^\infty \frac{(\pi s)^k}{k!} H^{(k)}. $
  \end{itemize}
\end{Cor}

\medskip
\begin{Cor}\label{C1-evol-ODE}  $  U(t, \vec b)=T_t(\vec b)$ is 
the solution to   the initial value problem   
$$\begin{cases} \frac d{dt} U(t,\vec a)= \pi H(U(t,\vec a)) & U\in C^1(\R, \ell^2)
\\
U(0, \vec a)=\vec b.
\end{cases}
$$
  \end{Cor}

Corollary \ref{C1-evol-ODE} may have application to signal processing.
 The following result is a consequence of the Hille-Yosida theorem.

\begin{Cor}\label{C4-Hille-Yos-H}
For every $\lambda >0$, the operator $\pi H-\lambda I $ is invertible in $\ell^2$, and 
  $|| (\pi H-\lambda I)^{-1}(\vec a)||_{\ell^2}<\frac {||\vec a||_{\ell^2}}\lambda $.  
\end{Cor}


\end{document}